\documentclass[11pt]{amsart}
\usepackage{amsmath,amsfonts,latexsym,graphicx,amssymb,url}
\usepackage{hyperref,pdfsync}
\usepackage{amsmath,txfonts,pifont,bbding,pxfonts,manfnt}
\usepackage[active]{srcltx}
\usepackage{wasysym,pstricks}
\usepackage{tikz}
\usepackage{xcolor}
\usepackage{soul}

\newcommand{\R}{{\mathbb R}} 

\renewcommand{\(}{\left(}
\renewcommand{\)}{\right)}
\newtheorem{theorem}{Theorem}
\newtheorem{corollary}[theorem]{Corollary}
\newtheorem{lemma}[theorem]{Lemma}

\newtheorem{definition}[theorem]{Definition}
\newtheorem{remark}[theorem]{Remark}

\begin{document}











\title[Differentiability of monotone maps related to non quadratic costs ~~~\today]{Differentiability of monotone maps\\ related to non quadratic costs}
\author[C. E. Guti\'errez and A. Montanari, \today]{Cristian E. Guti\'errez and Annamaria Montanari
\\
\today}
\thanks{
C.E.G. was partially supported by NSF grant DMS--1600578, and A. M. was partially supported by a grant from GNAMPA}
\address{Department of Mathematics\\Temple University\\Philadelphia, PA 19122}
\email{cristian.gutierrez@temple.edu}
\address{Dipartimento di Matematica\\Piazza di Porta San Donato 5\\Universit\`a di Bologna\\40126 Bologna, Italy}
\email{annamaria.montanari@unibo.it}

\maketitle

\begin{abstract}
The cost functions considered are $c(x,y)=h(x-y)$, with $h\in C^2\(\R^n\)$, homogeneous of degree $p\geq 2$, with positive definite Hessian in the unit sphere. We consider monotone maps $T$ with respect to that cost and establish local $L^\infty$-estimates of $T$ minus affine functions, which are applied to obtain differentiability properties of $T$ a.e. 
It is also shown that these maps are related
to maps of bounded deformation, and further differentiability and H\"older continuity properties are derived. \end{abstract}
\tableofcontents
\setstcolor{blue}

\setcounter{equation}{0}
\section{Introduction}
In this paper, we analyze the differentiability properties of monotone maps arising in optimal transport for cost functions that are not necessarily quadratic. 
We begin by proving $L^\infty$-estimates for these maps minus affine functions which are then applied to establish the differentiability of monotone maps using a notion of differentiability introduced by Calder\'on and Zygmund in their landmark paper \cite{Cal-Z}. Specifically, this notion of differentiability is preserved pointwise by singular integral operators.
This leads us also to a connection between monotone maps and maps of bounded deformation, a notion originated in elasticity \cite{Temam}.

The cost functions considered here have the form $c(x,y)=h(x-y)$ where $h\in C^2(\R^n)$, is nonnegative, positively homogeneous of degree $p$ for some $p\geq 2$, and satisfies \eqref{eq:strict ellipticity of D2h}. 
These costs include the power costs $|x-y|^p$ which are singular on the diagonal; in other words, $\dfrac{\partial^2 c}{\partial x\partial y}=0$ when $x=y$. Such a singularity prevents the application of results available for non-singular costs, such as those found in \cite{2020-otto-prodhomme-ried-general-costs}.
This work continues our research from \cite{Gutierrez-Montanari:Linfty-estimates} and 
\cite{Gutierrez-Montanari:ermanno-paper}, which originated from the fundamental work by Goldman and Otto \cite{2020-goldman-otto-variational}, who developed a variational approach to establish the partial regularity of optimal maps for quadratic costs.
The study of regularity of optimal maps for general costs has been a subject of important research, see for example the book \cite[Chapter 6]{ambrosio-gigli-savare-book}.

To introduce the notion of monotonicity, from optimal transport theory, if $c(x,y):D\times D^*\to [0,+\infty)$ is a general cost function, then 
the optimal map for the Monge problem is given by $T=\mathcal N_{c,\phi}$ where $\phi$ is $c$-concave and 
\[
\mathcal N_{c,\phi}(x)
=
\left\{m\in D^*:\phi(x)+\phi^c(m)=c(x,m)\right\}
\]
with $\phi^c(m)=\inf_{x\in D}\(c(x,m)-\phi(x)\)$, see for example \cite[Chapter 6]{Gutierrez:21}.
We say that a multivalued map $T:\R^n\to \mathcal P(\R^n)$ is $h$-monotone if
\begin{equation}\label{eq:map T is c-monotone h}
h\(x-\xi\)+ h\(y-\zeta\)\leq h\(x-\zeta\)+ h\(y-\xi\),
\end{equation}
for all $x,y\in \text{dom}(T)$ and for all $\xi\in T(x)$ and $\zeta\in T(y)$; where 
$\text{dom}(T)=\{x\in \R^n:T(x)\neq \emptyset\}$.
It is then clear that the optimal map $\mathcal N_{c,\phi}$ is $h$-monotone.

However, in our analysis, we will consider maps $T$ only satisfying \eqref{eq:map T is c-monotone h}; and {\it that are not necessarily optimal}. 
A challenge with \eqref{eq:map T is c-monotone h} is that $T$ is a multivalued map. However, we have proven in \cite{Gutierrez-Montanari:ermanno-paper} that for the costs considered $h$-monotone maps are single-valued a.e.. This proof is not trivial and therefore, the inequality \eqref{eq:map T is c-monotone h} holds a.e.

An outline of the paper is as follows.
Section \ref{sec:preliminaries} contains equivalent formulations of $h$-monotonicity, resembling the notion of standard monotone map, that are needed later.
Section \ref{sec:Linfty estimates} contains our first main result, Theorem \ref{thm:main Linfty estimate p-1}, concerning estimates of an $h$-monotone mapping minus affine functions.
Since the Calder\'on and Zygmund differentiability is related to approximation by polynomials, the estimates in Section \ref{sec:Linfty estimates} lead to differentiability properties of $h$-monotone maps which is the contents of Section \ref{sec:differentiability of h-monotone maps}.
Section \ref{subsec:bounded deformation} relates $h$-monotone maps with maps of bounded deformation, it shows that if $T$ is an $h$-monotone map, then the map $Dh(x-Tx)$ has bounded deformation and from \cite[Theorem 7.4]{ACDM} it is weakly differentiable a.e. As a consequence of the last result and the $L^\infty$-estimate of Theorem \ref{thm:main Linfty estimate p-1}, it is then proved in Section \ref{sec:Holder continuity} that $h$-monotone maps exhibit a H\"older continuity property a.e., Theorem \ref{Holder}. 
Finally, Section \ref{sec:appendix} is a brief appendix containing a result used in the proofs.

\section{Preliminaries on $h$-Monotone maps}\label{sec:preliminaries}
\setcounter{equation}{0}

In this section we present equivalent formulations of the Definition \eqref{eq:map T is c-monotone h} of $h$-monotonicity and integral representation formulas that will be needed to prove our results.
From \eqref{eq:map T is c-monotone h} and since $h\in C^2$ we have
\begin{align*}
0&\leq  h(y-\xi)-h(y-\zeta)-\(h(x-\xi)-h(x-\zeta)\)\notag\\
&=\int_0^1 \langle Dh(y-\zeta+s(\zeta-\xi)), \zeta-\xi\rangle ds -\int_0^1 \langle Dh(x-\zeta+s(\zeta-\xi)), \zeta-\xi\rangle ds\notag\\
&=\int_0^1 \langle Dh(y-\zeta+s(\zeta-\xi))-Dh((x-\zeta+s(\zeta-\xi)), \zeta-\xi\rangle ds\notag\\
&= \int_0^1\int_0^1\langle D^2h(y-\zeta+s(\zeta-\xi)+t (x-y)) (x-y),  (\xi-\zeta)\rangle dt\, ds\notag\\
&=  \langle A(x,y;\xi,\zeta) (x-y),  \xi-\zeta\rangle,\qquad \forall x,y\in \text{dom}(T), \xi\in T(x),\zeta\in T(y).
\end{align*}

Therefore \eqref{eq:map T is c-monotone h} is equivalent to 
\begin{equation}\label{eq:A}
 \langle A(x,y;\xi,\zeta) (x-y),  \xi-\zeta\rangle\geq 0,\qquad \forall x,y\in \text{dom}(T), \xi\in T(x),\zeta\in T(y)\footnote{When $h(x)=|x|^2$, this is the well-known notion of monotone map that has been the subject of large amount of research and applications, see for example the classical and polished book by H. Br\'ezis \cite{brezis-book-monotone-maps}.}
\end{equation}
with
\begin{equation}\label{eq:definition of A(x,y)}
A(x,y;\xi,\zeta)=\int_0^1\int_0^1 D^2h(y-\zeta+s(\zeta-\xi)+t (x-y))dt\, ds.
\end{equation}
The matrix $A(x,y;\xi,\zeta)$ is clearly symmetric, and satisfies $A(x,y;\xi,\zeta)=A(y,x;\zeta,\xi)$ by making the change of variables $t=1-t',s=1-s'$ in the integral.
If $h$ is homogenous of degree $p$ with $p\geq 2$, then $D^2h(z)$ is positively homogeneous of degree $p-2$.
We assume that $D^2h(x)$ is positive definite for each $x\in S^{n-1}$ and since $h\in C^2$, then there are positive constants $\lambda,\Lambda$ such that 
\begin{equation}\label{eq:strict ellipticity of D2h}
\lambda \,|v|^2
\leq 
\left\langle D^2h(x)v,v\right\rangle
\leq
\Lambda \,|v|^2,\qquad \forall x\in S^{n-1},v\in \R^n.
\end{equation}
We then have
\[
A(x,y;\xi,\zeta)=
\int_0^1 \int_0^1|y-\zeta+s(\zeta-\xi)+t (x-y)|^{p-2}D^2h\left(\frac{y-\zeta+s(\zeta-\xi)+t (x-y)}{ |y-\zeta+s(\zeta-\xi)+t (x-y)|}\right)dt \, ds
\]
and
\begin{equation}\label{eq:ellipticity of A}
\lambda \,\Phi(x,y;\xi,\zeta)\,|v|^2
\leq
\left\langle  A(x,y;\xi,\zeta)\,v,v\right\rangle\leq \Lambda \, \Phi(x,y;\xi,\zeta)\,|v|^2\quad \forall v \in \R^n,
\end{equation}
with 
\begin{equation}\label{eq:Phi}
 \Phi(x,y;\xi,\zeta)=\int_0^1\int_0^1|y-\zeta+s(\zeta-\xi)+t (x-y)|^{p-2}dt \, ds.
\end{equation}
We also have that $\Phi(x,y;\xi,\zeta)=0$ if and only if $y-\zeta+s(\zeta-\xi)+t (x-y)=0$ for all $s,t\in [0,1]$. That is, $\Phi(x,y;\xi,\zeta)=0$ if and only if $y-\zeta=0$, $\zeta-\xi=0$ and $x-y=0$.
Therefore $\Phi(x,y;\xi,\zeta)>0$ if and only if $\zeta\neq y$ or $\zeta\neq \xi$ or $x\neq y$.

If $T$ is $h$-monotone, then from \eqref{eq:map T is c-monotone h} we obviously have
{$
h\(y-Ty\)-h\(x-Ty\)\leq  h\(y-Tx\)- h\(x-Tx\).
$
Adding 
$
h(x-Ay-b)-h(y-Ay-b)
$, with $A$ any $n\times n$ matrix and $b\in \R^n$,
to the last inequality yields
\begin{align*}
&h\(y-Ty\)-h\(x-Ty\)
+h(x-Ay-b)-h(y-Ay-b)\\
&\leq  
h\(y-Tx\)-h(y-Ax-b)+h(x-Ax-b)- h\(x-Tx\)\\
&\qquad +h(y-Ax-b)-h(y-Ay-b)+h(x-Ay-b)-h(x-Ax-b).
\end{align*}
Therefore, if we set
\[
G(z_1,z_2,z_3)=h(z_2-z_3)-h(z_1-z_3)-h(z_2)+h(z_1),
\]
then \eqref{eq:map T is c-monotone h} is equivalent to
\begin{equation}\label{eq:main inequality to integrate}
G(x-Ay-b,y-Ay-b,Ty-Ay-b)
\leq
G(x-Ax-b,y-Ax-b,Tx-Ax-b)+P_{A,b}(x,y),
\end{equation}
where 
\[
P_{A,b}(x,y)=h(y-Ax-b)-h(y-Ay-b)+h(x-Ay-b)-h(x-Ax-b).
\]
To establish the desired estimate of $Tx-Ax-b$ in Theorem \ref{thm:main Linfty estimate p-1}, the idea is to use the potential theory formula \eqref{eq:third Green identity}
 with $v(x)$ equals the left hand side of \eqref{eq:main inequality to integrate} on a appropriate balls and majorize the resulting integrals using \eqref{eq:main inequality to integrate}. These integrations are possible due to the fact that the inequalities above hold a.e. which is a consequence of the fact proved in \cite[Theorem 1]{Gutierrez-Montanari:ermanno-paper} that $h$-monotone maps are single-valued a.e..
 Recall the formulas
\begin{equation}\label{eq:integral formula used many times}
h(\alpha)-h(\beta)=\int_0^1Dh\(\beta+s(\alpha-\beta)\) \cdot (\alpha-\beta) \, ds;
Dh(a)-Dh(b)=\int_0^1 D^2h\(b+\tau (a-b)\)(a-b)\,d\tau.
\end{equation}
We can then write the following integral representation formula for the function $G$:
\begin{align}\label{eq:integral representation for G}
G(z_1,z_2,z_3)
&=h(z_2-z_3)-h(z_1-z_3)-\(h(z_2)-h(z_1)\)\notag\\
&=
\int_0^1Dh\(z_1-z_3+s(z_2-z_1)\) \cdot (z_2-z_1) \, ds
-
\int_0^1Dh\(z_1+s(z_2-z_1)\) \cdot (z_2-z_1) \, ds\notag\\
&=
\int_0^1\left\{\int_0^1 D^2h\(z_1+s(z_2-z_1)-t z_3\)(-z_3)\,dt \right\} \cdot (z_2-z_1) \, ds\notag\\
&=
\int_0^1\int_0^1 \left\langle D^2h\(z_1+s(z_2-z_1)-t z_3\)z_3,  z_1-z_2\right\rangle \, ds\,dt.
\end{align}
Similarly, we obtain an integral expression for $P_{A,b}(x,y)$:
\begin{align}\label{eq:integral representation for H}
P_{A,b}(x,y)
&=
\int_0^1
\int_0^1 \left\langle D^2h\(x-Ay-b+s(Ay-Ax)+t(y-x)\)(y-x), Ay-Ax\right\rangle \, ds\,dt.
\end{align}

\section{$L^\infty$-estimates for $Tx-Ax-b$ with $T$ $h$-monotone}\label{sec:Linfty estimates}
\setcounter{equation}{0}
The purpose of this section is to prove the following local $L^\infty$-estimate for $h$-monotone maps.
\begin{theorem}\label{thm:main Linfty estimate p-1}
Let $p\geq 2$, $h\in C^2(\R^n)$, nonnegative, positively homogeneous of degree $p$, and satisfying \eqref{eq:strict ellipticity of D2h}. 
If $T\in L^{p-1}_{\text{loc}}(\R^n)$ is $h$-monotone, $A\in \R^{n\times n}$, $b\in \R^n$, $0<\beta<1$, and $u(x)=Tx-Ax-b$, then for each $x_0\in \R^n$ and $R>0$
\begin{equation}\label{eq:main L infty estimate general cost bis bis thm}
\sup_{x\in B_{\beta R}(x_0)}|u(x)|
\leq
\begin{cases}
K_1\,R\(R^{-(p-1)}\,\fint_{B_R(x_0)}|u(x)|^{p-1}\,dx\)^{1/(n+p-1)} & \text{if $\(\fint_{B_R(x_0)}|u(x)|^{p-1}\,dx\)^{1/(p-1)}\leq C(n,p,\beta)\,R$}\\
K_2\,R\,\(R^{-(p-1)}\,\fint_{B_R(x_0)}|u(x)|^{p-1}\,dx\)^{1/(p-1)}  & \text{if $\(\fint_{B_R(x_0)}|u(x)|^{p-1}\,dx\)^{1/(p-1)}\geq C(n,p,\beta)\,R$,}	
\end{cases}
\end{equation}
with the constant $K_1$ depending only on $n,p,\lambda,\Lambda$ in \eqref{eq:strict ellipticity of D2h}, and $\|A\|$; and with the constant $K_2$ depending only on the same parameters but also on $\beta$.

\end{theorem}
\begin{proof}
Our goal is to estimate the supremum of $|u|$ over a ball by the $L^{p-1}$-norm of $u$ over a slightly larger ball. To do this, the idea is to use the potential theory formula \eqref{eq:third Green identity} and estimate the integrals by integrating \eqref{eq:main inequality to integrate} in $x$.
The plan of the proof is as follows: 
\begin{enumerate}
\item Using \eqref{eq:third Green identity} we can write \eqref{eq:main formula A+B};
\item From \eqref{eq:G bigger than J} and Lemma \ref{lm:lower estimate of integral J}, a lower bound for the left hand side of \eqref{eq:main formula A+B} is proved in \eqref{Gatthecentermax};
\item Using the homogeneity of $h$ and \eqref{eq:strict ellipticity of D2h}, we estimate $II$ in \eqref{eq:main formula A+B} from above obtaining \eqref{eq:estimate of B};
\item next, using now that $T$ is $h$-monotone, i.e. \eqref{eq:main inequality to integrate}, an upper bound for $I$ in \eqref{eq:main formula A+B} follows, inequalities \eqref{A''} and \eqref{eq:estimate of A};
\item At this point we combine and balance the inequalities previously obtained for all terms in \eqref{eq:main formula A+B} obtaining the estimate \eqref{eq:estimate of u(y) simpler} for $|u|$ in a ball;
\item Finally, an estimate of the minimum in \eqref{eq:main estimate in r to iterate} yields the desired estimate.
\end{enumerate}

Let us set $\omega=\dfrac{u(y)}{|u(y)|}$ and $r=\delta \,|u(y)|$, with $\delta>0$ to be chosen; $u(y)\neq 0$.
Applying the identity \eqref{eq:third Green identity} with $v(x)\leadsto G(x-Ay-b,y-Ay-b,u(y))$ (a function that is $C^2$ in $x$ since $p\geq 2$) in the ball $B_r(y)\leadsto B_r(y+r\,\omega)$
yields
\begin{align}\label{eq:main formula A+B}
&v(y+r\,\omega)=G(y-Ay-b+ r\,\omega,y-Ay-b ,u(y))\notag\\
&=\fint_{B_r(y+r\,\omega)}G\(x-Ay-b,y-Ay-b,u(y)\)\,dx\notag \\
&\qquad +\dfrac{n}{r^n}\,
\int_0^r \rho^{n-1}\int_{|x-y-r \omega|\leq \rho}
\(\Gamma(x-y-r\,\omega)-\Gamma(\rho)\)\,\Delta_x G(x-Ay-b,y-Ay-b,u(y))\,
dx\,d\rho\notag\\
&=I+II.
\end{align} 

We first estimate the left hand side of \eqref{eq:main formula A+B}
from below.
From \eqref{eq:integral representation for G} write
\begin{align}\label{eq:G bigger than J}
&G(y-Ay-b+r\,\omega, y-Ay-b, u(y))\\
&= 
\int_0^1\int_0^1\left\langle D^2h\(y-Ay-b+r\omega- s\,r\,\omega-t\, u(y)\) u(y),r\omega \right\rangle \, ds \, dt\notag\\
&=\delta   \left\langle \int_0^1\int_0^1D^2h\(y-Ay-b+(1-s)\delta u(y)-t u(y)\)  \, ds \, dt  \, u(y), u(y)\right\rangle\,\text{ since $r\omega=\delta u(y)$}\notag\\
&=\delta   \left\langle \int_0^1\int_0^1D^2h\(y-Ay-b+ s\delta u(y)-tu(y)\)  \, ds \, dt  \, u(y), u(y)\right\rangle\text{ changing $1-s$ by $s$}\notag\\
&\geq \delta  \lambda  \left\langle \int_0^1\int_0^1|y-Ay-b-(t-s\delta)u(y)|^{p-2}  \, ds \, dt  \, u(y), u(y)\right\rangle \qquad \text{from \eqref{eq:strict ellipticity of D2h}}\notag\\
&\geq \delta  \lambda |u(y)|^2   \int_0^1\int_0^1|y-Ay-b-(t-s\delta)u(y))|^{p-2}  \, ds \, dt \notag.
\end{align}

The last double integral is estimated in the following lemma.

\begin{lemma}\label{lm:lower estimate of integral J}
Let $v_1,v_2$ be vectors in $\R^n$,  $p\geq 2$, $\delta>0$, and let
\begin{align*}
J=\int_0^1\int_0^1|v_1-(t-s\delta)v_2)|^{p-2}  \, ds \, dt.
\end{align*}
Then there exists $C_p,\delta_0>0$, both depending only on $p$, such that 
\begin{equation}\label{eq:lower estimate of J}
J\geq C_p\,\(\max\{|v_1|,|v_2|\}\)^{p-2},
\end{equation}
for $0<\delta\leq \delta_0$.
\end{lemma}
\begin{proof}
First make the change of variables $(\sigma,\tau)=\phi(s,t)=(t, t-s\delta)$, that is, $(s,t)=\phi^{-1}(\sigma,\tau)=\(\sigma,\frac{\sigma-\tau}{\delta}\)$, so $dsdt=\frac{1}{\delta}d\sigma d\tau$ and $R=\phi([0,1]^2)$ is the region bounded by the lines $\tau=\sigma$, $\sigma=1$, $\sigma=0$ and $\tau=\sigma-\delta$.
Hence 
\[
J=\frac{1}{\delta}\,\iint_R |v_1-\tau \,v_2|^{p-2}\,d\sigma d\tau
=\frac{1}{\delta}\,\int_0^1\(\int_{\sigma-\delta}^\sigma |v_1-\tau \,v_2|^{p-2}\,d\tau\) d\sigma.
\]
Suppose first that $|v_1|\geq |v_2|$. If $0\leq \sigma\leq 1$ and $\sigma-\delta\leq \tau\leq \sigma$, then $-\delta\leq \tau\leq 1$ ($0<\delta<1$) so $|\tau|\leq 1$, and $|v_1-\tau \,v_2|\geq ||v_1|-|\tau||v_2||=|v_1| \,\left|1-|\tau|\dfrac{|v_2|}{|v_1|}\right|=|v_1| \,\(1-|\tau|\dfrac{|v_2|}{|v_1|}\)\geq |v_1|\,\(1-|\tau|\)$. Hence
\begin{align*}
J&= \frac{1}{\delta}\,\int_0^1\(\int_{\sigma-\delta}^\sigma |v_1-\tau \,v_2|^{p-2}\,d\tau\) d\sigma\\
&\geq 
\frac{|v_1|^{p-2}}{\delta}\,\int_0^\delta\(\int_{\sigma-\delta}^\sigma \(1-|\tau |\)^{p-2}\,d\tau\) d\sigma
+
\frac{|v_1|^{p-2}}{\delta}\,\int_\delta^1\(\int_{\sigma-\delta}^\sigma \(1-|\tau |\)^{p-2}\,d\tau\) d\sigma\\
&=\frac{|v_1|^{p-2}}{\delta}\,I+\frac{|v_1|^{p-2}}{\delta}\,II
\end{align*}
We have 
\begin{align*}
I&= 
\int_0^\delta\(\int_{\sigma-\delta}^0 \(1-|\tau |\)^{p-2}\,d\tau
+
\int_0^\sigma \(1-|\tau |\)^{p-2}\,d\tau\) d\sigma\\
&=
\int_0^\delta\(\int_{\sigma-\delta}^0 \(1+\tau\)^{p-2}\,d\tau
+
\int_0^\sigma \(1-\tau\)^{p-2}\,d\tau\) d\sigma\\
&=
\dfrac{2\delta}{p-1}-\dfrac{2}{p(p-1)}+\dfrac{2(1-\delta)^p}{p(p-1)}.
\end{align*}
Also
\begin{align*}
II&= 
\int_\delta^1\int_{\sigma-\delta}^\sigma \(1-\tau\)^{p-2}\,d\tau\,d\sigma
=
\dfrac{1}{p(p-1)}-\dfrac{\delta^p}{p(p-1)}-\dfrac{(1-\delta)^p}{p(p-1)}.
\end{align*}
Then
\begin{align*}
I+II
=&
\dfrac{2\delta}{p-1}+\dfrac{(1-\delta)^p}{p(p-1)}-\dfrac{1}{p(p-1)}-\dfrac{\delta^p}{p(p-1)}\\
&\geq
\dfrac{2\delta}{p-1}-\dfrac{\delta}{p-1}-\dfrac{\delta^p}{p(p-1)}=\dfrac{\delta}{p-1}-\dfrac{\delta^p}{p(p-1)}.\end{align*}
Hence $\dfrac{1}{\delta}\(I+II\)\geq C_p>0$ for all $\delta<\delta_0$.
We then obtain 
\[
J\geq C_p\,|v_1|^{p-2}.
\]

Now suppose $|v_1|\leq |v_2|$.
Write $v_1=\alpha v_2+v_2^\perp$ where $v_2^\perp\perp v_2$, and by Pythagoras 
\[
|v_1-\tau \,v_2|^2=|(\alpha-\tau)v_2+v_2^\perp|^2=(\alpha-\tau)^2|v_2|^2+|v_2^\perp|^2
\geq (\alpha-\tau)^2|v_2|^2,
\]
so 
\begin{align*}
J&\geq \frac{1}{\delta}\,|v_2|^{p-2} 
\,
\int_0^1\(\int_{\sigma-\delta}^\sigma |\alpha-\tau|^{p-2}\,d\tau\) d\sigma\\
&=\frac{1}{\delta}\,|v_2|^{p-2} 
\,
\int_0^1\(\int_{-\delta}^1 \chi_{[\sigma-\delta,\sigma]}(\tau)\, |\alpha-\tau|^{p-2}\,d\tau\) d\sigma\\
&=\frac{1}{\delta}\,|v_2|^{p-2} 
\,
\int_{-\delta}^1|\alpha-\tau|^{p-2}\,|[0,1]\cap [\tau,\tau+\delta]| \,d\tau.
\end{align*}
Now 
\begin{equation*}
[0,1]\cap [\tau,\tau+\delta]
=
\begin{cases}
[0,\tau+\delta] & \text{if $-\delta\leq \tau\leq 0$}\\
[\tau,\tau+\delta] & \text{if $0\leq \tau\leq 1-\delta$}\\
[\tau,1] & \text{if $1-\delta\leq \tau\leq 1$}.
\end{cases}
\end{equation*}
Hence
\begin{align*}
&\int_{-\delta}^1|\alpha-\tau|^{p-2}\,|[0,1]\cap [\tau,\tau+\delta]| \,d\tau\\
&=
\int_{-\delta}^0|\alpha-\tau|^{p-2}\,(\tau+\delta) \,d\tau
+\int_{0}^{1-\delta}|\alpha-\tau|^{p-2}\,\delta \,d\tau
+\int_{1-\delta}^1|\alpha-\tau|^{p-2}\,(1-\tau) \,d\tau\\
&\geq
\delta \,\int_{0}^{1-\delta}|\alpha-\tau|^{p-2}\,d\tau
\geq \delta\,\int_{0}^{1/2}|\alpha-\tau|^{p-2}\,d\tau,\qquad \text{for $\delta<1/2$.}
\end{align*}
We have 
\[
\int_{0}^{1/2}|\alpha-\tau|^{p-2}\,d\tau= \int_{\alpha-1/2}^{\alpha}|s|^{p-2}\,ds\geq C_p>0\qquad \text{for all $\alpha\in \R$}.
\]
Therefore combining estimates we get 
\[
J\geq C_p \,|v_2|^{p-2}.
\]
and the lemma follows.
\end{proof}

Applying Lemma \ref{lm:lower estimate of integral J} with $v_1=y-Ay-b$ and $v_2=u(y)$ yields  from \eqref{eq:G bigger than J} that
\begin{equation}\label{Gatthecentermax} 
G\(y-Ay-b+r\,\omega, y-Ay-b, u(y)\)\geq \delta  \lambda  {C_p} |u(y)|^2 g(y)^{p-2},
\end{equation}
for $0<\delta<\delta_0$, where 
\[
g(y):=\max \Big\{ |y-Ay-b|, |u(y)|\Big\}.
\]

We now turn to estimate the terms $I$ and $II$ on the right hand side of \eqref{eq:main formula A+B}.
The first goal is to estimate $II$ proving \eqref{eq:estimate of B}.
From the definition of $G$ we have
\[
\Delta_x G(x-Ay-b,y-Ay-b,u(y))
=
\Delta h(x-Ay-b)-\Delta h(x-Ay-b-u(y)).
\]
Hence 
\begin{align*}
II
&=
\dfrac{n}{r^n}\int_0^r \rho^{n-1}
\,I(\rho,r,y)\,d\rho
\end{align*}
where 
\[
I(\rho,r,y)=\int_{B_\rho(y+r\omega)}\(\Gamma(x-y-r \omega)-\Gamma(\rho)\)\,
\(\Delta h(x-Ay-b)-\Delta h(x-Ay-b-u(y))\)\,dx.
\]
Making the change of variables $z=x-y-r\omega$ we have
\begin{align*}
I(\rho,r,y)=\int_{|z|\leq \rho}\(\Gamma(z)-\Gamma(\rho)\)\,
\(\Delta h(z+y+r\omega-Ay-b)-\Delta h(z+y+r\omega-Ay-b-u(y))\)\,dz.
\end{align*}
We have that $\Delta h$ is homogenous of degree $p-2$ so 
\begin{align*}
\Delta h(z+y+r\omega-Ay-b)&=
|z+y+r\omega-Ay-b|^{p-2}\,\Delta h\(\dfrac{z+y+r\omega-Ay-b}{|z+y+r\omega-Ay-b|}\).
\end{align*}
Write, with $e_1$ a fixed unit vector in $S^{n-1}$, 
{\small \begin{align*}
&\int_{|z|\leq \rho}\(\Gamma(z)-\Gamma(\rho)\)\,
\Delta h(z+y+r\omega-Ay-b)\,dz\\
&=\int_{|z|\leq \rho}\(\Gamma(z)-\Gamma(\rho)\)\,
|z+y+r\omega-Ay-b|^{p-2}\,\Delta h\(\dfrac{z+y+r\omega-Ay-b}{|z+y+r\omega-Ay-b|}\) \,dz\\
&=
\int_{|v|\leq \rho}\(\Gamma(v)-\Gamma(\rho)\)\,
|Ov+y-Ay-b+r\,Oe_1|^{p-2}\,\Delta h\(\dfrac{Ov+y-Ay-b+r\,Oe_1}{|Ov+y-Ay-b+r\,Oe_1|}\) \,dv,\,\text{$O$ rotation around $0$ with $Oe_1=\omega$}\\
&=
\int_{|v|\leq \rho}\(\Gamma(v)-\Gamma(\rho)\)\,
|Ov+Ov'+r\,Oe_1|^{p-2}\,\Delta h\(\dfrac{Ov+Ov'+r\,Oe_1}{|Ov+Ov'+r\,Oe_1|}\) \,dv,\quad y-Ay-b=Ov'\\
&=
\int_{|v|\leq \rho}\(\Gamma(v)-\Gamma(\rho)\)\,
|v+v'+r\,e_1|^{p-2}\,\Delta h\(\dfrac{O\(v+v'+r\,e_1\)}{|O\(v+v'+r\,e_1\)|}\) \,dv.
\end{align*}
}
{\small Similarly,
\begin{align*}
&\int_{|z|\leq \rho}\(\Gamma(z)-\Gamma(\rho)\)\,
\Delta h\(z+y+r\omega-Ay-b-u(y)\) \,dz\\
&=
\int_{|v|\leq \rho}\(\Gamma(v)-\Gamma(\rho)\)\,
|v+(r-|u(y)|)\,e_1+v'|^{p-2}\,\Delta h\(\dfrac{O\(v+(r-|u(y)|)\,e_1+v'\)}{|O\(v+(r-|u(y)|)\,e_1+v'\)|}\) \,dv,
\end{align*}
since $\omega=u(y)/|u(y)|$.
Thus
\begin{align*}
I(\rho,r,y)&=
\int_{|v|\leq \rho}\(\Gamma(v)-\Gamma(\rho)\)\,\\
&\qquad \(|v+v'+r\,e_1|^{p-2}\,\Delta h\(\dfrac{O\(v+v'+r\,e_1\)}{|O\(v+v'+r\,e_1\)|}\)
-
|v+(r-|u(y)|)\,e_1+v'|^{p-2}\,\Delta h\(\dfrac{O\(v+(r-|u(y)|)\,e_1+v'\)}{|O\(v+(r-|u(y)|)\,e_1+v'\)|}\)
\) \,dv.\end{align*}
We also have 
\begin{align*}
II&=
\dfrac{n}{r^n}\,\int_0^r \rho^{n-1}\,
I(\rho,r,y)\,d\rho=
n\,
\int_0^1t^{n-1}\,I(r\,t,r,y)\,dt.
\end{align*}
}
Now making the change of variables $v=r\zeta$ in the integral $I(rt,r,y)$ and setting $v'=r\zeta'$ yields
{\small \begin{align*}
&I(r\,t,r,y)\\
&=
\int_{|\zeta|\leq t}
\(\Gamma(r\,\zeta)-\Gamma(r\,t)\)\,\\
&\qquad \(|r\zeta+r\zeta'+r\,e_1|^{p-2}\,\Delta h\(\dfrac{O\(r\zeta+r\zeta'+r\,e_1\)}{|O\(r\zeta+r\zeta'+r\,e_1\)|}\)
-
|r\zeta+(r-|u(y)|)\,e_1+r\zeta'|^{p-2}\,\Delta h\(\dfrac{O\(r\zeta+(r-|u(y)|)\,e_1+r\zeta'\)}{|O\(r\zeta+(r-|u(y)|)\,e_1+r\zeta'\)|}\)
\) \,r^n\,d\zeta\\
&=r^p\,\int_{|\zeta|\leq t}
\(\Gamma(\zeta)-\Gamma(t)\)\,\\
&\qquad \(|\zeta+\zeta'+\,e_1|^{p-2}\,\Delta h\(\dfrac{O\(\zeta+\zeta'+\,e_1\)}{|O\(\zeta+\zeta'+\,e_1\)|}\)
-
|\zeta+(1-|u(y)|/r)\,e_1+\zeta'|^{p-2}\,\Delta h\(\dfrac{O\(\zeta+(1-|u(y)|/r)\,e_1+\zeta'\)}{|O\(\zeta+(1-|u(y)|/r)\,e_1+\zeta'\)|}\)
\) \,d\zeta
\end{align*}
}
and since $r=\delta\,|u(y)|$ we get 
\begin{align*}
II
&=
n\,
|u(y)|^p\,\delta^p\,\int_0^1t^{n-1}
\int_{|\zeta|\leq t}
\(\Gamma(\zeta)-\Gamma(t)\)\,\\
&\qquad \(|\zeta+\zeta'+\,e_1|^{p-2}\,\Delta h\(\dfrac{O\(\zeta+\zeta'+\,e_1\)}{|O\(\zeta+\zeta'+\,e_1\)|}\)
-
|\zeta+(1-1/\delta)\,e_1+\zeta'|^{p-2}\,\Delta h\(\dfrac{O\(\zeta+(1-1/\delta)\,e_1+\zeta'\)}{|O\(\zeta+(1-1/\delta)\,e_1+\zeta'\)|}\)
\) \,d\zeta\,dt\\
&=
n\,
|u(y)|^p\,\delta\,F(\delta,\zeta'),
\end{align*}
where 
\begin{align*}
F(\delta,\zeta')
&=
\delta^{p-1}\,
\int_0^1t^{n-1}
\int_{|\zeta|\leq t}
\(\Gamma(\zeta)-\Gamma(t)\)\,
 |\zeta+\zeta'+\,e_1|^{p-2}\,\Delta h\(\dfrac{O\(\zeta+\zeta'+\,e_1\)}{|O\(\zeta+\zeta'+\,e_1\)|}\)\,d\zeta\,dt\\
 &\qquad 
 -\delta\,
\int_0^1t^{n-1}
\int_{|\zeta|\leq t}
\(\Gamma(\zeta)-\Gamma(t)\)\,
|\delta\,\zeta+(\delta-1)\,e_1+\delta\,\zeta'|^{p-2}\,\Delta h\(\dfrac{O\(\delta\,\zeta+(\delta-1)\,e_1+\delta\,\zeta'\)}{|O\(\delta\,\zeta+(\delta-1)\,e_1+\delta\,\zeta'\)|}\)\,d\zeta\,dt\\
&=
\delta^{p-1}\,F_1(\delta,\zeta')-\delta\, F_2(\delta,\zeta')=\delta\,\(\delta^{p-2}\,F_1(\delta,\zeta')-F_2(\delta,\zeta')\).
\end{align*}
We claim that
\begin{equation}\label{eq:estimate of F(delta,zeta')}
|F(\delta,\zeta')|\leq C_{n,p}\,\Lambda\,\delta \, \( \frac{|y-Ay-b|}{|u(y)|}+1\)^{p-2},\qquad \delta<1/2.
\end{equation}
In fact, since $\Delta h$ is continuous, it is bounded in $S^{n-1}$. Also $Ov'=y-Ay-b$, $v'=r\,\zeta'$, $r=\delta\,|u(y)|$ and 
so $|\zeta'|=\dfrac{|y-Ay-b|}{\delta\,|u(y)|}$.
Thus
\[
|F_1(\delta,\zeta')|\leq C(n,p,\Lambda)\,(2+|\zeta'|)^{p-2}
\leq
C(n,p,\Lambda)\,\(\delta+\dfrac{|y-Ay-b|}{|u(y)|}\)^{p-2}\,\delta^{2-p}.
\]
Also 
\[
|F_2(\delta,\zeta')|\leq C(n,p,\Lambda)\,\(2+\delta\,|\zeta'|\)^{p-2}
\leq
C(n,p,\Lambda)\,\(2+\delta\,|\zeta'|\)^{p-2}
=
C(n,p,\Lambda)\,\(2+\dfrac{|y-Ay-b|}{|u(y)|}\)^{p-2}
\]
and then \eqref{eq:estimate of F(delta,zeta')} follows. This yields the following estimate of $II$:
\begin{align}\label{eq:estimate of B}
|II|
&\leq C(n,p,\Lambda)\, |u(y)|^p\,\delta^2\, \( \frac{|y-Ay-b|}{|u(y)|}+1\)^{p-2}\notag\\
&
\leq C(n,p,\Lambda)\, |u(y)|^2\,\delta^2\, g(y)^{p-2}
,\qquad \text{for $\delta<1/2$}.
\end{align}

We next estimate $I$ in \eqref{eq:main formula A+B}. 
Since $T$ is $h$-monotone, from \eqref{eq:main inequality to integrate}
\begin{align*}
I&=\fint_{B_r(y+r\,\omega)}G\(x-Ay-b,y-Ay-b,u(y)\)\,dx\\
&\leq
\fint_{B_r(y+r\,\omega)}G(x-Ax-b,y-Ax-b,u(x))\,dx+\fint_{B_r(y+r\,\omega)}P_{a,b}(x,y)\,dx
:=I'+I''.
\end{align*}

Let us estimate $I''$. From \eqref{eq:integral representation for H}, the $p-2$ homogeneity of $D^2h$, and Cauchy-Schwarz inequality it follows that
\begin{align*}
I''&=
\fint_{B_r(y+r\,\omega)}P_{A,b}(x,y)\,dx\\
&= \fint_{B_r(y+r\,\omega)} \int_0^1\int_0^1\left\langle D^2h\(y-Ax-b+t(Ax-Ay)+ s(x-y)\)A (x-y),(x-y) \right\rangle \, ds \, dt\,dx\\
&\leq \Lambda \,\|A\|
\fint_{B_r(y+r\,\omega)}\int_0^1 \int_0^1 |y-Ax-b+t(Ax-Ay)+ s(x-y)|^{p-2}\,|y-x|^2 ds\,dt\,dx\\
\\
&\leq \Lambda \|A\|
\fint_{B_r(y+r\,\omega)}\int_0^1 \int_0^1 |y-Ay-b+t(Ay-Ax)+ s(x-y)|^{p-2}\,|y-x|^2 ds\,dt\,dx,\quad \text{$t\leadsto 1-t$}\\
\\
&\leq \Lambda \|A\| C_p
\fint_{B_r(y+r\,\omega)}\( |y-Ay-b|^{p-2}+(\|A\|+1)^{p-2}\,|x-y)|^{p-2}\)\,|y-x|^2 \,dx\\
\\
&= \Lambda \|A\| C_p\(
|y-Ay-b|^{p-2}\fint_{B_r(y+r\,\omega)}|y-x|^2 \,dx+ (\|A\|+1)^{p-2}\fint_{B_r(y+r\,\omega)}\,|x-y|^{p}\) \,dx.\\
\end{align*}
If $x\in B_r(y+r\,\omega)$, then $|x-y|-r\leq |x-y-r\omega|\leq r$, that is, $|x-y|\leq 2r$ and we get
\begin{align}\label{A''}
I''&\leq \Lambda \|A\| C_p\(
|y-Ay-b|^{p-2}\, r^2+ (\|A\|+1)^{p-2}r^{p}\) .
\end{align}
Let us estimate $I'$. From \eqref{eq:integral representation for G}
with $z_1=x-Ax-b, z_2=y-Ax-b$ and $z_3=u(x)$ it follows as in the estimation of $I''$ that
{\small \begin{align*}
I'
&=
\fint_{B_r(y+r\,\omega)}\int_0^1 \int_0^1 \left\langle D^2h\(x-Ax-b-t\,u(x)+s\,(y-x)\)(y-x), -u(x)\right\rangle \,ds\,dt\,dx\\
&\leq
\Lambda 
\,\fint_{B_r(y+r\,\omega)}\int_0^1 \int_0^1 |x-Ax-b-t\,u(x)+s\,(y-x)|^{p-2}\,
|y-x|\,
|u(x)|
\,ds\,dt\,dx\\
&=
\Lambda\,\fint_{B_r(y+r\,\omega)} \,
|y-x|\,
|u(x)|\,\int_0^1 \int_0^1 |x-Ax-b-t\,u(x)+s\,(y-x)|^{p-2}
\,ds\,dt\,dx\\
&\leq
C_p\,\Lambda\,\fint_{B_r(y+r\,\omega)} \,
|y-x|\,
|u(x)|\,\(|x-Ax-b|^{p-2}+|u(x)|^{p-2}+|y-x|^{p-2}\)
\,dx\\
&=
C_p\,\Lambda\,
\(
\fint_{B_r(y+r\,\omega)} \,
|y-x|\,
|u(x)|\,|x-Ax-b|^{p-2}
\,dx
+
\fint_{B_r(y+r\,\omega)} \,
|y-x|\,
|u(x)|^{p-1}
\,dx
+
\fint_{B_r(y+r\,\omega)} \,
|u(x)|\,|y-x|^{p-1}
\,dx
\).
\end{align*}
}

Set
\begin{align*}
A_1
&=\fint_{B_r(y+r\,\omega)} \,
|y-x|\,
|u(x)|\,|x-Ax-b|^{p-2}
\,dx\\
A_2
&=
\fint_{B_r(y+r\,\omega)} \,
|y-x|\,
|u(x)|^{p-1}
\,dx\\
A_3
&=
\fint_{B_r(y+r\,\omega)} \,
|u(x)|\,|y-x|^{p-1}
\,dx.\end{align*}
If $x\in B_r(y+r\,\omega)$, then $|x-y|\leq 2r$.
Since $|Ax-Ay|\leq \|A\| \,|x-y|\leq 2r \|A\|$, we get
\begin{align*}
A_1
&\leq
2\,r\,\fint_{B_r(y+r\,\omega)} \,
|u(x)|\,|x-Ax-b|^{p-2}
\,dx
\leq
2\,r\,\(2r(1+\|A\|)+|y-Ay-b|\)^{p-2}\,\fint_{B_r(y+r\,\omega)} \,
|u(x)|
\,dx\\
A_2
&\leq
2\,r\, \fint_{B_r(y+r\,\omega)} \,
|u(x)|^{p-1}
\,dx\\
A_3
&\leq
(2r)^{p-1}\,
\fint_{B_r(y+r\,\omega)} \,
|u(x)|
\,dx.
\end{align*}
Consequently,
{\small \begin{equation}\label{eq:estimate of A}
I'
\leq
C_p\Lambda
\(r\,\(r(1+\|A\|)+|y-Ay-b|\)^{p-2}\,\fint_{B_r(y+r\,\omega)} \,
|u(x)|
\,dx
+
r\, \fint_{B_r(y+r\,\omega)} \,
|u(x)|^{p-1}
\,dx
+
r^{p-1}\,
\fint_{B_r(y+r\,\omega)} \,
|u(x)|
\,dx\).
\end{equation}}

Combining the estimates \eqref{Gatthecentermax}, \eqref{eq:estimate of B},  \eqref{eq:estimate of A}, and \eqref{A''} 
it follows that
\begin{align*}
&\delta\,\lambda\,C_p\,|u(y)|^2 g(y)^{p-2}\\
&\leq
C_{n,p}\, |u(y)|^2\,\delta^2\, g(y)^{p-2}\\
&\qquad +
C_p\Lambda
\(r\,\(r(1+\|A\|)+g(y)\)^{p-2}\,\fint_{B_r(y+r\,\omega)} \,
|u(x)|
\,dx
+
r\, \fint_{B_r(y+r\,\omega)} \,
|u(x)|^{p-1}
\,dx
+
r^{p-1}\,
\fint_{B_r(y+r\,\omega)} \,
|u(x)|
\,dx\)\\
&\qquad + \Lambda \|A\| C_p\(
g(y)^{p-2}\, r^2+ (\|A\|+1)^{p-2}r^{p}\) 
\end{align*}
for $0<\delta<1/2$, $\omega=u(y)/|u(y)|$, and $r=\delta|u(y)|$.
Dividing the last inequality by $\delta$ yields
\begin{align*}
&\lambda\,C_p\,|u(y)|^2\,g(y)^{p-2}\\
&\leq
C_{n,p}\, |u(y)|^2\,\delta\, g(y)^{p-2}\\
&+
\frac{C_p\Lambda}{\delta}
\(r\,\(r(1+\|A\|)+g(y)\)^{p-2}\,\fint_{B_r(y+r\,\omega)} \,
|u(x)|
\,dx+
r\, \fint_{B_r(y+r\,\omega)} \,
|u(x)|^{p-1}
\,dx
+
r^{p-1}\,
\fint_{B_r(y+r\,\omega)} \,
|u(x)|
\,dx\)\\
&\qquad + \frac{\Lambda \|A\| C_p}{\delta}
\(
g(y)^{p-2}\, r^2+ (\|A\|+1)^{p-2}r^{p}\)\\
&\leq
C_{n,p}\, |u(y)|^2\,\delta\, g(y)^{p-2}\\
&\qquad +
\frac{C_p\Lambda}{\delta}
\(\delta |u(y)|\,\(\delta (1+\|A\|) g(y)+g(y)\)^{p-2}\,\fint_{B_r(y+r\,\omega)} \,
|u(x)|
\,dx \right.\\
&\qquad \qquad \left.+
\delta |u(y)|\, \fint_{B_r(y+r\,\omega)} \,
|u(x)|^{p-1}
\,dx
+
\delta |u(y)|(\delta g(y))^{p-2}\,
\fint_{B_r(y+r\,\omega)} \,
|u(x)|
\,dx\)\quad \text{since $|u(y)|\leq g(y)$}\\
&\qquad +
\frac{\Lambda \|a\| C_p}{\delta}
\(
g(y)^{p-2}\, \delta^2\,|u(y)|^2+ (\|A\|+1)^{p-2}\delta^p\,|u(y)|^p\)\\
&\leq
C_{n,p}\, |u(y)|^2 \,g(y)^{p-2}\,\delta\\
&\qquad +
C_p\Lambda
\( |u(y)|\,g(y)^{p-2}\,\(\delta (1+\|A\|) +1\)^{p-2}\,\fint_{B_r(y+r\,\omega)} \,
|u(x)|
\,dx
\right.\\
&\qquad \qquad \qquad \left.+
 |u(y)|\, \(\frac{g(y)}{|u(y)|} \)^{p-2}\fint_{B_r(y+r\,\omega)} \, 
|u(x)|^{p-1}
\,dx
+
\delta^{p-2}\, |u(y)|\,g(y)^{p-2}\,
\fint_{B_r(y+r\,\omega)} \,
|u(x)|
\,dx\)\\
&\qquad +
\Lambda \|A\| C_p
\(
g(y)^{p-2}\, \delta\,|u(y)|^2+ (\|A\|+1)^{p-2}\delta^{p-1}\,|u(y)|^p\)
\end{align*}
where we have used $\frac{g(y)}{|u(y)|}\geq 1$.

Choosing $\delta_0=C(p,\lambda,\Lambda,\|A\|)>0$ sufficiently small, we then get for all $\delta\leq \delta_0$
that
{\small \begin{align*}
&(\lambda/2)\,C_p\,|u(y)|^2 \,g(y)^{p-2}\\
&\leq
C_p\Lambda
\(\(\(\delta (1+\|A\|)+1\)^{p-2}+\delta^{p-2}\)\,|u(y)| \,g(y)^{p-2}\,
\fint_{B_r(y+r\,\omega)} \,
|u(x)|
\,dx
+
|u(y)|\, \(\frac{g(y)}{|u(y)|} \)^{p-2} \fint_{B_r(y+r\,\omega)} \,
|u(x)|^{p-1}
\,dx
\)\\
&\qquad +
\Lambda \|A\| C_p
(\|A\|+1)^{p-2}\delta^{p-1}\,|u(y)|^p.
\end{align*}}
Multiplying this inequality by $\dfrac{|u(y)|^{p-3}}{g(y)^{p-2}}$ yields
{\small \begin{align*}
|u(y)|^{p-1}
&\leq
C(p,\lambda,\Lambda, \|A\|)
\(|u(y)|^{p-2}\,
\fint_{B_r(y+r\,\omega)} \,
|u(x)|
\,dx
+
\fint_{B_r(y+r\,\omega)} \,|u(x)|^{p-1}\,dx+
\delta^{p-2}\,|u(y)|^{p-2}\,
\fint_{B_r(y+r\,\omega)} \,
|u(x)|
\,dx\)\\
&\qquad +
C(\Lambda,p,\|A\|)\, \dfrac{\delta^{p-1}\,|u(y)|^{p-1}\,|u(y)|^{p-2}}{g(y)^{p-2}}\\
&\leq
C(p,\lambda,\Lambda, \|A\|)
\(|u(y)|^{p-2}\,
\fint_{B_r(y+r\,\omega)} \,
|u(x)|
\,dx
+
\fint_{B_r(y+r\,\omega)} \,|u(x)|^{p-1}\,dx+
\delta^{p-2}\,|u(y)|^{p-2}\,
\fint_{B_r(y+r\,\omega)} \,
|u(x)|
\,dx\)\\
&\qquad +
C(\Lambda,p,\|A\|)\, \delta^{p-1}\,|u(y)|^{p-1},
\end{align*}}
since $|u(y)|\leq g(y)$.
Hence, since $p>1$, we can choose $\delta_0$ even smaller than before, depending only on $\Lambda,p$ and $\|A\|$, so that the following estimate holds for $0<\delta<\delta_0$
{\small \begin{align*}
|u(y)|^{p-1}
&\leq
C(p,\lambda,\Lambda, \|A\|)
\(|u(y)|^{p-2}\,
\fint_{B_r(y+r\,\omega)} \,
|u(x)|
\,dx
+
\fint_{B_r(y+r\,\omega)} \,|u(x)|^{p-1}\,dx+
\delta^{p-2}\,|u(y)|^{p-2}\,
\fint_{B_r(y+r\,\omega)} \,
|u(x)|
\,dx\),
\end{align*}}
with $r=\delta\,|u(y)|$.
If $p=2$ we then get the inequality
\[
|u(y)|
\leq
C(p,\lambda,\Lambda, \|A\|)
\fint_{B_r(y+r\,\omega)} \,
|u(x)|
\,dx.
\]
When $p>2$ we write
\begin{align*}
|u(y)|^{p-2}\,
\fint_{B_r(y+r\,\omega)} \,
|u(x)|
\,dx
&\leq
\dfrac{p-2}{p-1}\(\dfrac{1}{\epsilon}\)^{(p-1)/(p-2)}|u(y)|^{p-1}+
\dfrac{1}{p-1}\(\epsilon\,\fint_{B_r(y+r\,\omega)} \,
|u(x)|
\,dx\)^{p-1}.
\end{align*}
Choosing $\epsilon>0$ large so that
$$
C(p,\lambda,\Lambda, \|A\|)\dfrac{p-2}{p-1}\(\dfrac{1}{\epsilon}\)^{(p-1)/(p-2)}\leq 1/2
$$
we then obtain the inequality
\begin{align}\label{eq:main inequality with delta}
|u(y)|^{p-1}
&\leq
C
\(\(
\fint_{B_r(y+r\,\omega)} \,
|u(x)|
\,dx\)^{p-1}
+
\fint_{B_r(y+r\,\omega)} \,|u(x)|^{p-1}\,dx
+
\delta^{p-2}\,|u(y)|^{p-2}\,
\fint_{B_r(y+r\,\omega)} \,
|u(x)|
\,dx\)\notag\\
&\leq
C
\(
\fint_{B_r(y+r\,\omega)} \,|u(x)|^{p-1}\,dx
+
\delta^{p-2}\,|u(y)|^{p-2}\,
\(\fint_{B_r(y+r\,\omega)} \,
|u(x)|^{p-1}
\,dx\)^{1/(p-1)}\)\notag,
\end{align}
from H\"older's inequality, $p-1\geq 1$.
Since $r=\delta\,|u(y)|$ we get
\begin{align*}
|u(y)|^{p-1}
&\leq
C
\(
\fint_{B_r(y+r\,\omega)} \,|u(x)|^{p-1}\,dx
+
r^{p-2}\,
\(\fint_{B_r(y+r\,\omega)} \,
|u(x)|^{p-1}
\,dx\)^{1/(p-1)}\).
\end{align*}
By Young's inequality 
\[
r^{p-2}\,
\(\fint_{B_r(y+r\,\omega)} \,
|u(x)|^{p-1}
\,dx\)^{1/(p-1)}
\leq
\dfrac{1}{p-1}
\(\(\fint_{B_r(y+r\,\omega)} \,
|u(x)|^{p-1}
\,dx\)^{1/(p-1)}\)^{p-1}
+
\dfrac{p-2}{p-1}\,\(r^{p-2}\)^{(p-1)/(p-2)},
\]
and so we obtain the estimate
\begin{equation}\label{eq:bound for u(y) to the p-1 delta_0}
|u(y)|^{p-1}
\leq
C\,\(\fint_{B_r(y+r\,\omega)} \,
|u(x)|^{p-1}
\,dx+r^{p-1}\),
\end{equation}
with a constant $C$ depending only on $n, p,\lambda,\Lambda$ and $\|A\|$ for $\delta=\frac{r}{|u(y)|}<\delta_0$.

Let us now fix a ball $B_R(x_0)$, and
suppose $y\in B_{\beta\,R}(x_0)$ with $0<\beta<1$, $R>0$. Then 
$B_r(y+r\,\omega) \subset B_R(x_0)$ for $r\leq \dfrac{1-\beta}{2}R$ and so
\[
\fint_{B_r(y+r\,\omega)}|u(x)|^{p-1}\,dx
\leq
\dfrac{C_n}{r^n}\,\int_{B_R(x_0)}|u(x)|^{p-1}\,dx.
\]

Given any $0<r\leq \dfrac{1-\beta}{2}R$, let $\delta=\dfrac{r}{|u(y)|}$.
If $\delta\leq \delta_0$, we then obtain from \eqref{eq:bound for u(y) to the p-1 delta_0}
\begin{align}\label{eq:estimate of u(y) simpler}
|u(y)|^{p-1}
&\leq
C
\(
\dfrac{1}{r^n}\,\int_{B_R(x_0)} \,|u(x)|^{p-1}\,dx
+
r^{p-1}\) \,:=H(r).
\end{align}
If $\delta\geq \delta_0$, then
\[
|u(y)|\leq \dfrac{r}{\delta_0}.
\]
Hence
\[
|u(y)|\leq \max \left\{H(r)^{1/(p-1)},\dfrac{r}{\delta_0} \right\}:=\bar H(r)
\]
for any $0<r \leq (1-\beta)R/2$, $y\in B_{\beta R}(x_0)$.
This means that 
\[
|u(y)|\leq \min_{0<r \leq (1-\beta)R/2} \bar H(r).
\]
Since \eqref{eq:estimate of u(y) simpler} holds with any constant larger than $C$, we take $C$ sufficiently large such that $C\geq 1/\delta_0^{p-1}$, in this way $H(r)^{1/(p-1)}\geq \dfrac{r}{\delta_0}$, and so $\max \left\{H(r)^{1/(p-1)}, \dfrac{r}{\delta_0} \right\}=H(r)^{1/(p-1)}$.
Therefore we obtain the estimate 
\begin{equation}\label{eq:main estimate in r to iterate}
|u(y)|
\leq 
\min_{0<r\leq (1-\beta)R/2} H(r)^{1/(p-1)}\qquad \text{for $y\in B_{\beta\,R}(x_0)$}.
\end{equation}
Next we estimate the minimum on the right hand side of \eqref{eq:main estimate in r to iterate}. Set 
\[
\Delta=\int_{B_R(x_0)}|u(x)|^{p-1}\,dx,
\]
so $H(r)= C\,\(\Delta\,r^{-n}+r^{p-1}\) $. The minimum of $H$ over $(0,\infty)$ is attained at 
\[
r_0=\(\dfrac{n\,\Delta}{p-1}\)^{1/(n+p-1)},
\]
$H$ is decreasing in $(0,r_0)$ and increasing in $(r_0,\infty)$, and
\[
\min_{[0,\infty)}H(r)=H(r_0)=
C\,\(\(\dfrac{n}{p-1}\)^{-n/(n+p-1)}+
\(\dfrac{n}{p-1}\)^{(p-1)/(n+p-1)}\)\,\Delta^{(p-1)/(n+p-1)}.
\]
If $r_0<(1-\beta)R/2$, then $\min_{0<r\leq (1-\beta)R/2} H(r)=H(r_0)$. 
On the other hand, if $r_0\geq (1-\beta)R/2$, that is, 
$\Delta\geq \(\dfrac{1-\beta}{2}R\)^{n+p-1} \dfrac{p-1}{n} :=\Delta_0$, then we have
\begin{align*}
\min_{0<r<(1-\beta)R/2} H(r)&=H\(\dfrac{1-\beta}{2}R\)
=
C\,\(\Delta\,\(\dfrac{1-\beta}{2}R\)^{-n}+
\(\dfrac{1-\beta}{2}R\)^{p-1}\)\\
&\leq
C\,\(\Delta\,\(\dfrac{1-\beta}{2}R\)^{-n}
+
\dfrac{n}{p-1}\,\Delta\,\(\dfrac{1-\beta}{2}R\)^{-n}\)\\
&=
C\,\dfrac{p+n-1}{p-1}\,\(\dfrac{1-\beta}{2}R\)^{-n}\Delta:=K_2\,R^{-n}\,\Delta.
\end{align*}
We then obtain the following estimate valid for all $0<\beta<1$ for $y\in B_{\beta\,R}(x_0)$
\begin{equation}\label{eq:main L infty estimate general cost}
|u(y)|^{p-1}
\leq
\begin{cases}
K_1\,\Delta^{(p-1)/(n+p-1)} & \text{if $\Delta\leq \Delta_0$}\\
K_2\,R^{-n}\,\Delta  & \text{if $\Delta\geq \Delta_0$,}	
\end{cases}
\end{equation}
with $K_1= C\,\(\(\dfrac{n}{p-1}\)^{-n/(n+p-1)}+
\(\dfrac{n}{p-1}\)^{(p-1)/(n+p-1)}\)$, 
$K_2=C\,\dfrac{p+n-1}{p-1}\,\(\dfrac{1-\beta}{2}\)^{-n}$, 
$\Delta= \int_{B_R(x_0)}|u(x)|^{p-1}\,dx$, and $\(\dfrac{1-\beta}{2}R\)^{n+p-1} \dfrac{p-1}{n} :=\Delta_0$.

We rewrite the last inequality in the following form
\begin{equation}\label{eq:main L infty estimate general cost bis}
\sup_{x\in B_{\beta R}(x_0)}|u(x)|
\leq
\begin{cases}
K_1\,R^{n/(n+p-1)}\(\fint_{B_R(x_0)}|u(x)|^{p-1}\,dx\)^{1/(n+p-1)} & \text{if $\(\fint_{B_R(x_0)}|u(x)|^{p-1}\,dx\)^{1/(p-1)}\leq C(n,p,\beta)\,R$}\\
K_2\,\(\fint_{B_R(x_0)}|u(x)|^{p-1}\,dx\)^{1/(p-1)}  & \text{if $\(\fint_{B_R(x_0)}|u(x)|^{p-1}\,dx\)^{1/(p-1)}\geq C(n,p,\beta)\,R$,}	
\end{cases}
\end{equation}
with $u(x)=Tx-Ax-b$, $p\geq 2$, and with the constant $K_1$ depending only on $n,p,\lambda,\Lambda$ and $\|A\|$; the constant $K_2$ depends only on the same parameters but also on $\beta$.

The inequality \eqref{eq:main L infty estimate general cost bis}
can be rewritten in the form \eqref{eq:main L infty estimate general cost bis bis thm} which resembles \cite[inequality (2.5)]{Gutierrez-Montanari:Linfty-estimates}.
This completes the proof of Theorem \ref{thm:main Linfty estimate p-1}.
\end{proof}

\section{Application to the differentiability of $h$-monotone maps}\label{sec:differentiability of h-monotone maps}
\setcounter{equation}{0}
Following Calder\'on and Zygmund \cite{Cal-Z}, see also \cite[Sect. 3.5]{Zim}, we recall the notion of differentiability in $L^p$-sense.
\begin{definition}
Let $1\leq p\leq \infty$, $k$ be a real number $\geq -n/p$ and $f\in L^p(\Omega)$, with $\Omega\subset \R^n$ open, and let $x_0\in \Omega$.
We say that $f\in T^{k,p}(x_0) \(f\in t^{k,p}(x_0)\)$\footnote{Sometimes we also use of the letter $T$ to denote a vector valued map. The expression $T\in T^{k,p}(x_0)$   means that the map $T$ belongs to the class $T^{k,p}(x_0)$, that is, each component belong to the class.}  if there exists a polynomial $P_{x_0}$ of degree strictly less than $k \(\text{$P_{x_0}$ of degree }\leq k\)$\footnote{If $k\leq 0$, $P_{x_0}\equiv 0$.} such that 
\begin{align*}
&\(\fint_{B_r(x_0)} |f(x)-P_{x_0}(x)|^p\,dx\)^{1/p}=O(r^k)\quad \text{as $r\to 0$}\\
&\(\(\fint_{B_r(x_0)} |f(x)-P_{x_0}(x)|^p\,dx\)^{1/p}=o(r^k)\quad \text{as $r\to 0$}\);
\end{align*}
when $p=\infty$ the averages are replaced by $\text{\rm ess sup}_{x\in B_r(x_0)}|f(x)-P_{x_0}(x)|=\|f-P_{x_0}\|_{L^\infty \(B_r(x_0)\)}$.
\end{definition}
From Theorem \ref{thm:main Linfty estimate p-1} we then get the following two corollaries.
\begin{corollary}
Let $T$ be $h$-monotone. If $T\in L^{p-1}\(B_\epsilon(x_0)\)$ and there exists a matrix $A=A_{x_0}$ and a vector $b=b_{x_0}$ such that 
\[
\(\fint_{B_R(x_0)}|Tx-Ax-b|^{p-1}\,dx\)^{1/(p-1)}=o(R)\quad \text{as $R\to 0$,}
\]
that is, $T\in t^{1,p-1}(x_0)$, then $T$ is differentiable in the ordinary sense at $x_0$.
\end{corollary}
\begin{proof}
Let $u(x)=Tx-Ax-b$, and $\phi(R)=\(\fint_{B_R(x_0)}|u(x)|^{p-1}\,dx\)^{1/(p-1)}$. From the assumption $\phi(R)/R\to 0$ as $R\to 0$, 
and so from \eqref{eq:main L infty estimate general cost bis}
\[
\sup_{B_{\beta R}(x_0)}|u(x)|
\leq
K_1\,R^{n/(n+p-1)}\,\phi(R)^{(p-1)/(n+p-1)}
=
C\,R\,\(\dfrac{\phi(R)}{R}\)^{(p-1)/(n+p-1)},
\]
that is, $T$ is differentiable at $x_0$.
\end{proof}}
\begin{corollary}\label{thm:differentiability of monotone maps}
Let $T$ be an $h$-monotone map in $L_{\text{loc}}^{p-1}\(\R^n\)$, $p\geq 2$, and such that
\begin{equation*}\label{eq:L2-integrability of T bis}
\(\fint_{B_R(x_0)}|Tx-b|^{p-1}\,dx\)^{1/(p-1)}=O\(R\)\quad \text{as $R\to 0$} 
\end{equation*}
for some vector $b=b_{x_0}$, i.e, $T\in T^{1,p-1}(x_0)$ for all $x_0$ in a measurable set $E$.
Then  
\[
\|Tx-A(x-x_0)-Tx_0\|_{L^\infty\(B_R(x_0)\)}=o\(R\)\quad \text{as $R\to 0$}
\] 
for almost all $x_0\in E$ and some $A=A_{x_0}\in \R^{n\times n}$, i.e., $T\in t^{1,\infty}(x_0)$ for a.e. $x_0\in E$.

\end{corollary} 
\begin{proof}
For each $x_0\in E$ there exist constants $M(x_0)\geq 0$, $R_0>0$ and $b\in \R^n$ such that
\[
\(\fint_{B_R(x_0)}|Tx-b|^{p-1}\,dx\)^{1/(p-1)}\leq M(x_0)\,R
\]
for all $0<R<R_0$, i.e., $T\in T^{1,p-1}(x_0)$.
Given $0<R<R_0$ we have either $\(\fint_{B_R(x_0)}|u(x)|^{p-1}\,dx\)^{1/(p-1)}\leq C(n,p,\beta)\,R$ or $\(\fint_{B_R(x_0)}|u(x)|^{p-1}\,dx\)^{1/(p-1)}\geq C(n,p,\beta)\,R$.
In the first case, 
from \eqref{eq:main L infty estimate general cost bis}
\[
\sup_{B_{\beta R}(x_0)}|Tx-b|\leq
K_1\,R^{n/(n+p-1)}\,\(\fint_{B_R(x_0)}|Tx-b|^{p-1}\,dx\)^{1/(n+p-1)}
\leq
C(n,\beta)\,R.
\]
In the second case,
\[
\sup_{B_{\beta R}(x_0)}|Tx-b|\leq
K_2\,\(\fint_{B_R(x_0)}|Tx-b|^{p-1}\,dx\)^{1/(p-1)}
\leq
K_2\,M(x_0)\,R.
\]
This means $\sup_{B_{R}(x_0)}|Tx-b|=O(R)$ as $R\to 0$ for all $x_0\in E$, i.e., $T\in T^{1,\infty}(x_0)$.
By Stepanov's theorem \cite[Chap. VIII, Thm. 3, p. 250]{Stein} this implies that 
$Tx$ is differentiable for a.e. $x_0\in E$, i.e., $T\in t^{1,\infty}(x_0)$ for a.e. $x_0\in E$.
\end{proof}

\begin{remark}\rm
In the particular case that $T$ is the optimal transport map, from \cite[Theorem 6.2.7]{ambrosio-gigli-savare-book} $T$ is approximately differentiable. 
\end{remark}

\setcounter{equation}{0}
\section{$h$-monotone maps and bounded deformation}\label{subsec:bounded deformation}

\begin{definition}
A mapping $u:\R^n\to \R^n$ is of bounded deformation, denoted $u\in BD(\R^n)$, if $u\in L^1_{\text{loc}}(\R^n)$ and the symmetrized gradient $\dfrac12\(\dfrac{\partial u}{\partial x}+\(\dfrac{\partial u}{\partial x}\)^t\)$, as a matrix valued distribution, can be represented by a matrix valued signed Radon measure.  
\end{definition}
This notion appears in elasticity, see \cite{Temam}.
It is proved in \cite[Theorem 4.5]{Gutierrez-Montanari:Linfty-estimates} that if $T$ is a monotone map in the standard sense, i.e., $h$-monotone with $h(x)=|x|^2$, with $T\in L^1_\text{loc}$, then $T\in BD(\R^n)$ and it is differentiable a.e. (see earlier results by \cite{Mig}).
For general $h$-monotone maps we have the following.

\begin{theorem}\label{thm:h-monotone is bounded deformation}
Let $T$ be an $h$-monotone map with $T\in L^{p-1}_\text{loc}(\R^n)$ and $h$ satisfying the assumptions of Theorem \ref{thm:main Linfty estimate p-1}. 
Then the map $Dh(x-Tx)$ is of bounded deformation and belongs to $t^{1,1}(x_0)$
for a.e. $x_0\in \R^n$. 
\end{theorem}
\begin{proof}
Applying the monotonicity with $y=x+t\xi$
\[
0\leq h(x+t\xi-Tx)-h(x-Tx)+h(x-T(x+t\xi))-h(x+t\xi -T(x+t\xi)).
\]
From the estimates in Theorem \ref{thm:main Linfty estimate p-1}, it follows that the function $h(x-Tx)$ is locally bounded.
Hence multiplying the last inequality by $\phi\in C_0^\infty(\R^n)$, $\phi\geq 0$ we get using \eqref{eq:integral formula used many times}
\begin{align*}
0
&\leq
\int_{\R^n}
\left[ h(x+t\xi-Tx)-h(x-Tx)\right]\,\phi(x)\,dx+
\int_{\R^n}
\left[ h(z-t\xi-Tz)-h(z -Tz)\right]\,\phi(z-t\xi)\,dz\\
&=
\int_{\R^n}
\left[ \int_0^1 Dh(x-Tx+st\xi)\cdot (t\xi)\,ds\right]\,\phi(x)\,dx+
\int_{\R^n}
\left[ \int_0^1 Dh(z-Tz-st\xi)\cdot (-t\xi)\,ds\right]\,\phi(z-t\xi)\,dz\\
&=
\int_{\R^n}
\int_0^1 \left[ Dh(x-Tx+st\xi)-Dh(x-Tx-st\xi)\right]\cdot (t\xi)\,ds\,\phi(x)\,dx\\
&\qquad
-
\int_{\R^n}
\left[ \int_0^1 Dh(x-Tx-st\xi)\cdot (t\xi)\,ds\right]\,\(\phi(x-t\xi)-\phi(x)\)\,dx\\
&=
\int_{\R^n}
\int_0^1 \left[ \int_0^1 D^2h\(x-Tx-st\xi+\tau(2st\xi)\)(2st\xi)\,d\tau\right]\cdot (t\xi)\,ds\,\phi(x)\,dx\\
&\qquad
-
\int_{\R^n}
\left[ \int_0^1 Dh(x-Tx-st\xi)\cdot (t\xi)\,ds\right]\,\(\phi(x-t\xi)-\phi(x)\)\,dx\\
&=
t^2\,\(2\,A-B\),
\end{align*}
with 
\begin{align*}
A
&=
\int_{\R^n}
\int_0^1 \left[ \int_0^1 D^2h\(x-Tx-st\xi+\tau(2st\xi)\)(s\xi)\,d\tau\right]\cdot \xi\,ds\,\phi(x)\,dx\\
B
&=
\int_{\R^n}
\left[ \int_0^1 Dh(x-Tx-st\xi)\cdot \xi\,ds\right]\,\dfrac{\phi(x-t\xi)-\phi(x)}{t}\,dx.
\end{align*}
We calculate the limits of $A,B$ when $t\to 0$:
\begin{align*}
&\int_0^1  \int_0^1 D^2h\(x-Tx-st\xi+\tau(2st\xi)\)s\,d\tau\,ds\\
&\to \int_0^1  \int_0^1 D^2h\(x-Tx\)s\,d\tau\,ds=\frac12\,D^2h\(x-Tx\)
\end{align*}
as $t\to 0$ on compact subsets of $\R^n$ ($Tx$ is locally bounded), since $D^2h$ is uniformly continuous on compact sets.
Hence
\[
A\to \frac12 \, \int_{\R^n} \left\langle D^2h\(x-Tx\)\xi,\xi\right\rangle\,\phi(x)\,dx.
\]
And also
\[
B\to 
-\int_{\R^n}
\(Dh(x-Tx)\cdot \xi \)\,\partial_\xi \phi(x)\,dx.
\]
We then obtain the inequality
\begin{equation}\label{eq:generalized bounded deformation}
\int_{\R^n} \left\langle D^2h\(x-Tx\)\xi,\xi\right\rangle\,\phi(x)\,dx
+
\int_{\R^n}
\(Dh(x-Tx)\cdot \xi \)\,\partial_\xi \phi(x)\,dx\geq 0
\end{equation}
for all $\phi\in C_0^\infty$, $\phi\geq 0$, and each unit vector $\xi$.
Now 
\begin{align*}
\int_{\R^n}
\(Dh(x-Tx)\cdot \xi \)\,\partial_\xi \phi(x)\,dx
=
-\sum_{i,j=1}^n\left\langle \frac12\(\dfrac{\partial h_{x_i}(x-Tx)}{\partial x_j}+\dfrac{\partial h_{x_j}(x-Tx)}{\partial x_i}\),\phi\right\rangle\,\xi\,\xi_j
\end{align*}
with derivatives in the sense of distributions,
and similarly
\begin{align*}
&\int_{\R^n} \left\langle D^2h\(x-Tx\)\xi,\xi\right\rangle\,\phi(x)\,dx
=
\sum_{i,j=1}^n \left\langle h_{x_i x_j}(x-Tx),\phi\right\rangle \,\xi_i\,\xi_j.
\end{align*}
Therefore, if $a_{ij}$ is the scalar distribution defined by 
\begin{equation}\label{eq:definition of entries of a}
a_{ij}=h_{x_i x_j}(x-Tx)-\frac12\(\dfrac{\partial h_{x_i}(x-Tx)}{\partial x_j}+\dfrac{\partial h_{x_j}(x-Tx)}{\partial x_i}\),
\end{equation}
then the matrix valued distribution $a=\(a_{ij}\)$ is non negative and so it can be represented with a matrix-valued Radon measure, see \cite[Thm. 2.1.7]{LH:book on pdes vol I}\footnote{As usual, a measure is a nonnegative (in the sense needed) $\sigma$-additive set function, and a signed measure is a  $\sigma$-additive set function that can take positive and negative values.}. 
Since $T$ is locally bounded and $h\in C^2$, then matrix $D^2 h(x-Tx)$ defines a matrix valued distribution of order zero. Therefore the matrix 
\[
\frac12\,\( \dfrac{\partial Dh(x-Tx)}{\partial x}+\(\dfrac{\partial Dh(x-Tx)}{\partial x}\)^t\)
\]
can be represented in the sense of distributions with a matrix valued signed Radon measure, that is, $Dh(x-Tx)$ is of bounded deformation as desired.
We can now apply \cite[Theorem 7.4]{ACDM} (see also \cite{Haj}) to conclude that the mapping $Dh(x-Tx)\in t^{1,1}(x_0)$
for a.e. $x_0\in \R^n$.

\end{proof}
\color{black}

\subsection{H\"older continuity of $h$-monotone maps}\label{sec:Holder continuity}
The purpose of this section is to show that $h$-monotone maps $T$ that are locally in $L^{p-1}$ satisfy the H\"older type condition $T\in T^{1/(p-1),\infty}$ which is the contents of Theorem \ref{Holder}. This will follow as a consequence of Theorems \ref{thm:h-monotone is bounded deformation} and \ref{thm:main Linfty estimate p-1}.
We begin with a couple of lemmas.
\begin{lemma}\label{lemma2}
Let $v_1,v_2$ be vectors in $\R^n$,  $p\geq 2$,  and let
\begin{align*}
J=\int_0^1|v_1+tv_2|^{p-2}  \, dt.
\end{align*}
Then there exist constants $c_p, C_p>0$, both depending only on $p$, such that 
\begin{equation}\label{eq:lower estimate of J}
c_p\,\(\max\{|v_1|,|v_2|\}\)^{p-2}\leq J\leq C_p\,\(\max\{|v_1|,|v_2|\}\)^{p-2},
\end{equation}

\end{lemma}

\begin{proof} The proof is similar to the proof of Lemma \ref{lm:lower estimate of integral J}.
By triangle inequality and the monotonicity of $f(r)=r^{p-2}$ for $r>0,$ we have
\[
J=\int_0^1|v_1+tv_2|^{p-2}  \, dt\leq \int_0^1\left(|v_1|+t|v_2|\right)^{p-2}  \, dt\leq \(\max\{|v_1|,|v_2|\}\)^{p-2}\int_0^1\left(1+t\right)^{p-2} \, dt.
\]
To prove the estimate of $J$ from below, suppose first that $|v_1|\geq |v_2|$. Hence, again
by triangle inequality and the monotonicity $f(r)=r^{p-2}$ for $r>0,$ 
\begin{align*}
J&= \int_0^1 |v_1+t\,v_2|^{p-2}\,dt \geq \int_0^1 \big| |v_1|-t\,|v_2|\big|^{p-2}\,dt
\geq 
|v_1|^{p-2}\,\int_0^1 \(1-t\)^{p-2}\,dt,
\end{align*}
obtaining  
$
J\geq C_p\,|v_1|^{p-2}$.
Now suppose $|v_1|\leq |v_2|$.
Writing $v_1=\alpha v_2+v_2^\perp$ with $v_2^\perp\perp v_2$ yields 
$
|\alpha| =\frac{ |v_1 \cdot v_2| }{|v_2|^2}\leq \frac{ |v_1| \,|v_2| }{|v_2|^2}\leq 1$,
and 
$
|v_1+t \,v_2|^2=|(\alpha+t)v_2+v_2^\perp|^2=(\alpha+t)^2|v_2|^2+|v_2^\perp|^2
\geq (\alpha+t)^2|v_2|^2
$.
So, making the changing variables $\alpha+t=\tau$ we obtain
\begin{align*}
J\geq |v_2|^{p-2} 
\,
\int_0^1 |\alpha+t|^{p-2}\,dt=\,|v_2|^{p-2} 
\,
\int_\alpha^{\alpha +1}|\tau|^{p-2}\,d\tau &\geq \,|v_2|^{p-2} 
\,\min_{\alpha \in[-1,1]}\left(
\int_\alpha^{\alpha +1}|\tau|^{p-2}\,d\tau\right)\\
&=\,|v_2|^{p-2}
\int_{-1/2}^{1/2}|\tau|^{p-2}\,d\tau.
\end{align*}
Therefore,  
$
J\geq c_p \,|v_2|^{p-2}
$
and the lemma follows.
\end{proof}
\begin{lemma}\label{lemma3}
There are positive constants $c_p,C_p$ such that for all $a,b\in \R^n$ 
we have 
\begin{equation*}
{\lambda c_p}|a-b|^p\leq \left(Dh(a)-Dh(b)\right)\cdot (a-b)\leq {\Lambda C_p}|a-b|^{2}\max\Big\{{|b|^{p-2}}, {|a-b|^{p-2}}\Big\},
\end{equation*}
and consequently 
\begin{equation}\label{second bis}
|Dh( a)-Dh( b)|\geq  \lambda c_p|a-b|^{p-1}.
\end{equation}

\end{lemma}
\begin{proof}
From \eqref{eq:integral formula used many times}, \eqref{eq:strict ellipticity of D2h}, and Lemma \ref{lemma2} we have
\begin{align*}
\left(Dh(a)-Dh(b)\right)\cdot (a-b) &= \int_0^1 \langle D^2h(b+t(a-b))(a-b), a-b \rangle \,  dt \\
&\leq \Lambda |a-b|^2\int_0^1|b+t(a-b)|^{p-2} \, dt
\leq \Lambda C_p|a-b|^{2}\max \{|b|^{p-2}, |a-b|^{p-2}\},
\end{align*}
and similarly the opposite inequality follows.

\end{proof}
\begin{theorem}\label{Holder}
If $T$ is an $h$-monotone map with $T\in L^{p-1}_\text{loc}(\R^n)$ and $h$ satisfying the assumptions of Theorem \ref{thm:main Linfty estimate p-1}, then  
\[
\sup_{B_{R}(x_0)}|Tx-Tx_0|=O(R^{1/(p-1)})\quad \text{as $R\to 0$} 
\]
for almost all $x_0\in \R^n$;
that is, $T\in T^{1/(p-1),\infty}(x_0)$ for a.e. $x_0$.\end{theorem}
\begin{proof}

By Theorem \ref{thm:h-monotone is bounded deformation} we know that the mapping $Dh(x-Tx)\in t^{1,1}(x_0)$. Hence $Dh(x-Tx)\in T^{1,1}(x_0)$ and consequently at each Lebesgue point $x_0$ of the function $Dh(x-Tx)$ it follows that
\begin{equation*}\label{eq1}
\fint_{B_R(x_0)}|Dh(x-Tx) - Dh(x_0-Tx_0) |\,dx=O\(R\)\quad \text{as $R\to 0$}, 
\end{equation*}
that is, for a.e. $x_0\in \R^n$.
In fact, since $Dh(x-Tx)\in T^{1,1}(x_0)$, there is a polynomial $P_{x_0}(x)$ of degree less than 1, i.e, a constant such that $\fint_{B_R(x_0)}|Dh(x-Tx) - P_{x_0}(x) |\,dx=O\(R\)$. If $x_0$ is a Lebesgue point of $Dh(x-Tx)$, then $\fint_{B_R(x_0)}|Dh(x-Tx) - Dh(x_0-Tx_0) |\,dx=O(1)$. Hence
\begin{align*}
&\fint_{B_R(x_0)}|P_{x_0}(x) - Dh(x_0-Tx_0) |\,dx\\
&\leq
\fint_{B_R(x_0)}|Dh(x-Tx) - Dh(x_0-Tx_0) |\,dx+\fint_{B_R(x_0)}|P_{x_0}(x) - Dh(x-Tx) |\,dx\\
&=O(1)+O(R)
\end{align*}
and so $P_{x_0}(x) - Dh(x_0-Tx_0)=0$ and we are done.

Applying triangle inequality and inequality \eqref{second bis} with $a=Tx-Tx_0$ and $ b=x-x_0$ yields
\begin{align*}
\fint_{B_R(x_0)}|Tx-Tx_0|^{p-1}\, dx&\leq \fint_{B_R(x_0)}(|Tx-Tx_0-(x-x_0)|+ |x-x_0|)^{p-1} \,dx\\
&\leq 2^{p-1} \fint_{B_R(x_0)}|Tx-Tx_0-(x-x_0)|^{p-1}\,dx + 2^{p-1}\fint_{B_R(x_0)} |x-x_0|^{p-1} \,dx\\
& \leq \frac{2^{p-1}}{\lambda c_p} \fint_{B_R(x_0)}|Dh(x-Tx) - Dh(x_0-Tx_0) |\,dx+ O\(R^{p-1}\)\quad \text{as $R\to 0$} \\
&=O\(R\)\quad \text{as $R\to 0$,} 
\end{align*}
at a.e. $x_0\in \R^n$, since $p\geq 2$.
Hence for a.e. $x_0$ there exist constants $M(x_0)\geq 0$ and $R_0>0$  such that
\[
\(\fint_{B_R(x_0)}|Tx-Tx_0|^{p-1}\,dx\)^{1/(p-1)}\leq M(x_0)\,R^{1/(p-1)}
\]
for all $0<R<R_0.$
We now apply Theorem \ref{thm:main Linfty estimate p-1} to $u(x)= Tx-Tx_0$ with $0<R<R_0$. 
We have either $\(\fint_{B_R(x_0)}|u(x)|^{p-1}\,dx\)^{1/(p-1)}\leq C(n,p,\beta)\,R$ \,  or \, $\(\fint_{B_R(x_0)}|u(x)|^{p-1}\,dx\)^{1/(p-1)}\geq C(n,p,\beta)\,R$.
In the first case, 
we get from \eqref{eq:main L infty estimate general cost bis}
\[
\sup_{B_{\beta R}(x_0)}|Tx-Tx_0|\leq
K_1\,R^{n/(n+p-1)}\,\(\fint_{B_R(x_0)}|Tx-Tx_0|^{p-1}\,dx\)^{1/(n+p-1)}
\leq
C'(n,p,\beta)\,R.
\]
In the second case,
\[
\sup_{B_{\beta R}(x_0)}|Tx-Tx_0|\leq
K_2\,\(\fint_{B_R(x_0)}|Tx-Tx_0|^{p-1}\,dx\)^{1/(p-1)}
\leq
K_2\,M(x_0)\,R^{1/(p-1)}.
\]
Therefore, $\sup_{B_{R}(x_0)}|Tx-Tx_0|=O(R^{1/(p-1)})$ as $R\to 0$ for all a.e. $x_0\in \R^n$ as desired.

\end{proof}
\section{Appendix}\label{sec:appendix}
\setcounter{equation}{0}
Recall that $\Gamma(x)=\dfrac{1}{n\omega_n (2-n)}|x|^{2-n}$ with $n>2$, where $\omega_n$ is the volume of the unit ball in $\R^n$, and  
the Green's representation formula
\[
v(y)=\int_{\partial \Omega} \(v(x)\,\dfrac{\partial \Gamma}{\partial \nu}(x-y)-\Gamma(x-y)\,\dfrac{\partial v}{\partial \nu}(x)\)\,d\sigma(x)+\int_\Omega \Gamma(x-y)\,\Delta v(x)\,dx
\]
where $\nu$ is the outer unit normal and $y\in \Omega$.
If $\Omega=B_\rho(y)$, then 
$\dfrac{\partial \Gamma}{\partial \nu}(x-y)=\dfrac{1}{n\,\omega_n}\,|x-y|^{1-n}$ and so the representation formula reads
\begin{align*}
v(y)
&=
\fint_{|x-y|=\rho}v(x)\,d\sigma(x)-\Gamma(\rho)\,\int_{|x-y|=\rho} \dfrac{\partial v}{\partial \nu}(x)\,d\sigma(x)
+\int_{|x-y|\leq \rho} \Gamma(x-y)\,\Delta v(x)\,dx\\
&=
\fint_{|x-y|=\rho}v(x)\,d\sigma(x)
+\int_{|x-y|\leq \rho} \(\Gamma(x-y)-\Gamma(\rho)\)\,\Delta v(x)\,dx
\end{align*}
from the divergence theorem.
Multiplying the last identity by $\rho^{n-1}$ and integrating over $0\leq \rho\leq r$ yields
\begin{equation}\label{eq:third Green identity}
v(y)=\fint_{|x-y|\leq r}v(x)\,dx+\dfrac{n}{r^n}\,\int_0^r \rho^{n-1}\int_{|x-y|\leq \rho}\(\Gamma(x-y)-\Gamma(\rho)\)\,\Delta v(x)\,dx\,d\rho.
\end{equation}


%

\end{document}